\documentclass[reqno,12pt]{amsart}
\usepackage[normalem]{ulem}
\usepackage{amssymb}
\usepackage{hyperref}
\usepackage{graphicx}
\usepackage{empheq}
\usepackage{pstricks}
\hypersetup{colorlinks,linkcolor={black},citecolor={blue},urlcolor={blue}}

\newtheorem{lem}{Lemma}[section]

\newtheorem{thm}{Theorem}[section]

\newtheorem{df}{Definition}[section]
\newtheorem{rmq}{Remark}[section]

\newtheorem{obs}{\sc Remark}[section]
\newcommand\norm[1]{\left\| #1\right\|}

\theoremstyle{definition}

\theoremstyle{remark}

\numberwithin{equation}{section}

\newcommand{\R}{\mathbb{R}}

\makeatletter
\renewcommand{\author}[2][]{%
  \def\@tempa{#1}
  \ifx\@empty\authors
    \ifx\@tempa\@empty
      \gdef\shortauthors{#2}%
    \else
      \gdef\shortauthors{#1}%
    \fi
    \gdef\authors{\author{#2}}%
  \else
    \ifx\@tempa\@empty
      \g@addto@macro\shortauthors{\and#2}%
    \else
      \g@addto@macro\shortauthors{\and#1}%
    \fi
    \g@addto@macro\authors{\and\author{#2}}%
  \fi
}
\renewcommand{\address}[2][]{\g@addto@macro\authors{\address{#1}{#2}}}
\def\@setauthors{%
  \begin{center}%
    \footnotesize
    \vspace{20pt}
    \let\and\@empty
    \def\author##1{\advance\@tempcnta\@ne}%
    \def\address##1##2{\advance\@tempcntb\@ne}%
    \@tempcnta=\z@  \@tempcntb=\z@
    \authors
    \ifnum\@tempcnta>\@ne \ifnum\@tempcntb=\@ne
        \oneaddress
      \else
        \sepaddresses
      \fi
    \else
      \oneaddress
    \fi
  \end{center}%
}
\def\oneaddress{%
  \begingroup
  \let\author\@iden \let\address\@gobbletwo
  \renewcommand{\andify}{%
    \nxandlist{\unskip, }{\unskip{} and~}{\unskip, and~}}%
  \uppercasenonmath\authors
  \andify\authors
  \authors
  \endgroup
  \begingroup \let\and\relax \let\author\@gobble
  \def\address##1##2{\unskip\\[10pt] \itshape##2}%
  \authors
  \endgroup
}
\def\sepaddresses{%
  \begingroup
    \baselineskip10\p@\relax
    \def\address##1##2{ ({\itshape##2}\/)}
    \def\author##1{\def\temp{##1}\leavevmode\uppercasenonmath\temp\temp}%
    \nxandlist
      {,\\[\baselineskip]}
      {\\[\baselineskip] \textsc{\lowercase{and}}\\[\baselineskip]}
      {,\\[\baselineskip]\textsc{\lowercase{and}}\\[\baselineskip]}
      \authors % macro to operate on
    \authors
  \endgroup
}
\def\maketitle{\par
  \@topnum\z@
  \@setcopyright
  \thispagestyle{firstpage}%
  \uppercasenonmath\shorttitle
  \ifx\@empty\shortauthors \let\shortauthors\shorttitle
  \else
    \newcommand{\@xuppercasenonmath}[1]{\toks@\@emptytoks
      \@xp\@skipmath\@xp\@empty##1$$%
      \edef##1{\@nx\protect\@nx\@upprep\the\toks@}}%
    \@xuppercasenonmath\shortauthors
    \def\@@and{AND}
    \renewcommand{\andify}{%
      \nxandlist{\unskip, }{\unskip{ }\@@and{ }}{\unskip, \@@and{ }}}%
    \andify\shortauthors
  \fi
  \@maketitle@hook
  \begingroup
  \@maketitle
  \endgroup
  \c@footnote\z@
  \@cleartopmattertags
}
\def\@maketitle{%
  \normalfont\normalsize
  \let\@makefntext\noindent
  \@adminfootnotes
  \ifx\@empty\addresses\else \@footnotetext{\@setotheraddresses}\fi
  \global\topskip68\p@\relax
  \@settitle
  \ifx\@empty\authors \else \@setauthors \fi
  \ifx\@empty\@dedicatory
  \else
    \baselineskip26\p@
    \vtop{\centering{\footnotesize\itshape\@dedicatory\@@par}%
      \global\dimen@i\prevdepth}\prevdepth\dimen@i
  \fi
  \toks@\@xp{\shortauthors}\@temptokena\@xp{\shorttitle}%
  \edef\@tempa{\@nx\markboth{\the\toks@}{\the\@temptokena}}\@tempa
  \@setabstract
  \normalsize
  \if@titlepage
    \newpage
  \else
    \dimen@34\p@ \advance\dimen@-\baselineskip
    \vskip\dimen@\relax
  \fi
} % end \@maketitle
\renewcommand{\thanks}[1]{%
  \ifx\@empty\thankses
    \gdef\thankses{\thanks{#1}}%
  \else
    \g@addto@macro\thankses{\endgraf\thanks{#1}}%
  \fi}
\def\@setthanks{\def\thanks##1{\noindent##1\@addpunct.}\thankses}
\renewcommand{\curraddr}[2][]{%
  \ifx\@empty\addresses
    \gdef\addresses{\curraddr{#1}{#2}}%
  \else
    \g@addto@macro\addresses{\endgraf\curraddr{#1}{#2}}%
  \fi}
\renewcommand{\email}[2][]{%
  \ifx\@empty\addresses
    \gdef\addresses{\email{#1}{#2}}%
  \else
    \g@addto@macro\addresses{\endgraf\email{#1}{#2}}%
  \fi}
\renewcommand{\urladdr}[2][]{%
  \ifx\@empty\addresses
    \gdef\addresses{\urladdr{#1}{#2}}%
  \else
    \g@addto@macro\addresses{\endgraf\urladdr{#1}{#2}}%
  \fi}
\def\@setotheraddresses{%
  \def\curraddr##1##2{\noindent
    \emph{Current address\@ifnotempty{##1}{ of ##1}}:\space
      ##2\@addpunct.}%
  \def\email##1##2{\noindent
    \emph{E-mail address\@ifnotempty{##1}{ of ##1}}:\space
      \texttt{##2}}%
  \def\urladdr##1##2{\noindent
    \emph{WWW address\@ifnotempty{##1}{ of ##1}}:\space
      \texttt{##2}}%
  \addresses
}
\let\enddoc@text\relax
\makeatother

\begin{document}
\title[Approximate controllability of semi-linear heat equation]{Approximate controllability of semi-linear heat equation with non instantaneous impulses, memory and delay }

\author[H. LEIVA, W. Zouhair, M.N. Entekhabi and L. Delgado]{H. Leiva$^{1}$, W. Zouhair$^{2}$, M. N. Entekhabi$^{3}$ and E. L. Delgado$^{4}$  }

\address{$^{1}$ School of Mathematical Sciences and Information Technology,Universidad Yachay Tech, 
         San Miguel de Urcuqui, Ecuador, \\hleiva@yachaytech.edu.ec, hleiva@ula.ve}
         
\address{$^{2}$ Laboratory of Mathematics and Population Dynamics
Marrakesh
        BP 2390, Marrakesh 40000, Morocco, Cadi Ayyad University, Faculty of Sciences Semlalia, Maroc\\walid.zouhair.fssm@gmail.com} 
\address{ $^{3}$ Florida A and M University
         Department of Mathematics 
        Tallahassee, Florida 32307- USA\\ mozhgan.entekhabi@famu.edu}

\address{ $^{4}$ Faculty of Sciences, Polytechnic School of Chimborazo, Riobamba, Ecuador, \\ elucena@espoch.edu.ec}

\begin{abstract}
The semilinear heat equation with non instantaneous impulses \textbf{(NII)}, memory and delay is considered and its approximate controllability is obtained. This is done by employing a technique that avoids fixed point theorems and pulls back the control solution to a fixed curve in a short time interval. We demonstrate, once again, that the controllability of the system is robust under the influence of  non instantaneous impulses, memory and delays. Finally, we present some open problems and a possible general framework to study the controllability of non instantaneous impulses semilinear systems.
\end{abstract}
%%
%%  LaTeX will not make the title for the paper unless told to do so.
%%  This is done by uncommenting the following.
%%

\maketitle

%% LaTeX can automatically make a table of contents.  This is done by
%% uncommenting the following:
%%

%\tableofcontents

%%
%%  To enter text is easy.  Just type it.  A blank line starts a new
%%  paragraph.
%%
\section{Introduction}
The theory of impulsive dynamical systems was initiated by V.D. Mil'man and A. Mishkis in 1960 \cite{Milman}. Afterwards, it has become an important field of investigation in several areas.
They can be found in applications ranging from neural networks, ecology, chemistery, biotechnology, radiophysics, theoretical physics, mathematical economy and engineering. The interest of this article is the non-instantaneous impulsive
semi-linear system involving memory and state-delay, which is motivated by applications, such as species population, nanoscale electronic circuits consisting of single-electron tunneling junctions, and mechanical systems with impacts \cite{Lakshmikantham,Samoilenko,Yang}. In general, impulses represent sudden deviations of the states at specific times, by either instantaneous jumps or continuous intervals.

In real life problems, the impulse starts abruptly at a certain moment of time and remains active on a finite time interval. However, the time of the action is little. Such an impulse is known as non-instantaneous impulse. This notion appears for the first time in 2012. After that, it has become an area of interest for many researchers. For more, we refer to our readers \cite{SKSMA,VKMM, RPASHDO,Zouhair,Zouhair2}.

The phenomenon of impulses implies instantaneous and discontinuous changes at different instants of time, which influence the solutions and can lead to the instability (respectively uncontrollability) of the differential equation or conversely to its stability (respectively controllability), which explains the evolution of this theory has been rather
slow due to the complexity of handling such equations (see \cite{liu}, \cite{farzana} ). Afterwards, many scientists contributed in the enrichment of this theory, they launched different studies on this subject and  large number of results were established .

Controllability is a mathematical problem, which consists of finding a controls steering the system from an arbitrary initial state to a final state in a finite interval of time, the controllability of instantaneous impulsive systems have been extensively studied in the literature, see \cite{SQG,KDP,GL,DULC,Zouhair3,Zouhair4}. To the best of authors’s knowledge, there is no paper which deals with semilinear heat equation with memory and delay in the presence of non instantaneous impulsive. Motivated by the above facts, we study the approximate controllability for the following system
{\footnotesize
\begin{empheq}[left = \empheqlbrace]{alignat=2}
\begin{aligned}\label{1.1}
&\frac{\partial \omega}{\partial{t}} + A \omega = \mathbf{1}_{\theta} u(t,x)+ \int_{0}^{t} M(t,s)  && \text{in } \bigcup\limits_{i=0}^{N}\left(s_{i}, t_{i+1}\right] \times [0,\pi],\\
& g(\omega(s-r,x))\, \mathrm{d} s + f(t,\omega(t-r,x),u(t,x)),\\
&\omega(t,0) = \omega(t,\pi)= 0,  && \text{on}\quad (0, T),\\
&\omega(s,x) = h(s,x), &&\text{in}\quad [-r,0] \times [0,\pi],\\
&\omega(t,x)= G_{i}(t,\omega(t,x),u(t,x)),  && \text{in}\quad \bigcup\limits_{i=0}^{N}\left(t_{i}, s_{i}\right]\times [0,\pi],
\end{aligned}
\end{empheq}  
}
where $0=s_{0}=t_{0} <t_{1} \leq s_{1} < ...< t_{N} \leq s_{N} < t_{N+1} = T $ are fixed real numbers,
$h: \,[-r,0] \times [0,\pi]\longrightarrow \R$ is a piecewise continuous function in $s$, $f:[0,T] \times \R \times \R \longrightarrow \R$ represents the non-linear
perturbation of the differential equation in the system  and the non-instantaneous impulses are represented by $G_{i}:\left(t_{i}, s_{i}\right] \times \R \times \R \longrightarrow \R$. $A: \mathcal{D}(A) \subset \mathcal{X} \longrightarrow \mathcal{X}$ is the operator $A \psi = - \psi_{xx}$ with domain $
\mathcal{D}(A) := \lbrace \psi \in \mathcal{X} :\psi, \psi_{x} \mbox{ absolutely continuous}, \psi_{xx} \in \mathcal{X},\,\, \psi(0)=\psi(\pi) = 0 \rbrace,
$ such that $\mathcal{X} = L^{2}[0,\pi]$ and $(\mathcal{D}(A))^{1/2} = \mathcal{X}^{1/2}$,  $\theta$ is an open nonempty subset of $[0,\pi],$  $\mathbf{1}_{\theta}$ denotes the characteristic function of the set $\theta.$

In general, the effect of such pulses in the behaviour of solutions is presented as follows
\begin{center}
\includegraphics[scale=0.6]{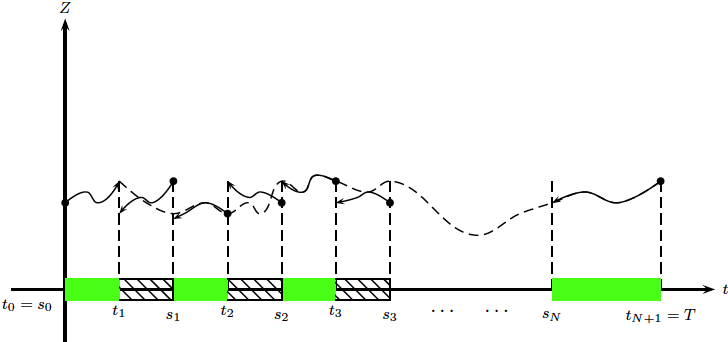}\\
\textit{Example of the effect of non-instantaneous pulses \textbf{(NII)}.}\par\medskip
\end{center}
where the green part of the figure represents the impulsive behavior, it starts at instant $s_k$ and remains active until the instant $t_k$.

This paper is organized as follows. In section \ref{sec2}, we briefly present the problem formulation and related definition. In section \ref{sec3}  and \ref{sec4},  we discuss the approximation controlability for the linear and the semilinear system. The last section is devoted to some related open problems and application.

\section{Abstract Formulation of the Problem}\label{sec2}
In this section, we recall some results that will be useful in the sequel. It is well known that  $-A: \mathcal{D}(A) \subset \mathcal{X} \longrightarrow \mathcal{X},$ is the generator of a  strongly continuous analytic semigroup  $(S(t))_{t\geq 0}$ on $\mathcal{X}.$ Moreover, the operator $A$ and the semigroup $(S(t))_{t\geq 0}$ can be represented as follows.
$$
 Ax = \sum_{n = 1}^{\infty} \lambda_{n} \langle x, \phi_{n} \rangle \phi_{n}, \ \ x \in \mathcal{X},
$$
where $\lambda_{n}= n^2$, $\phi_{n}(\xi) = \sin(n \xi)$ and $\langle \cdot, \cdot \rangle$ is the inner product in $\mathcal{X}$. Also, the strongly continuous semigroup $\left(S(t)\right)_{t\geq 0}$ generated by $A$ is compact and presented by
$$
S(t)x =  \sum_{n = 1}^{\infty} e^{- n^2 t}\langle x, \phi_{n} \rangle \phi_{n}, \ \ x \in \mathcal{X}.
$$
Then, we have the following estimation.
$$
\parallel S(t)\parallel\leq e^{-t},\quad t\geq0.
$$

On the other hand, we rewrite system \eqref{1.1} as an abstract differential equations with memory as follows.
\begin{empheq}[left = \empheqlbrace]{alignat=2}
\begin{aligned}\label{2.1}
&\frac{\partial \omega}{\partial{t}} + A \omega = \mathbf{B}_{\theta} u+ \int_{0}^{t} M(t,s)  && \text{in } \bigcup\limits_{i=0}^{N}\left(s_{i}, t_{i+1}\right],\\
& g^{1}(\omega_{s}(-r)) \mathrm{d} s\,+ f^{1}(t,\omega_{t}(-r),u(t)),\\
&\omega(s) = h(s), &&\text{in}\quad [-r,0],\\
&\omega(t,x)=  G_{i}^{1}(t,\omega(t),u(t)),  && \text{in}\quad \bigcup\limits_{i=0}^{N}\left(t_{i}, s_{i}\right],
\end{aligned}
\end{empheq} 
where $u \in L^{2}([0, T] ;\mathcal{U}),$ such that $\mathcal{U}= \mathcal{X}$, $ B_{\theta}: \mathcal{U} \longrightarrow \mathcal{X} $ is a bounded linear operator defined by
$B_{\theta} u = \mathbf{1}_{\theta} u$,  $\omega_{t}$ stands for the translated function of  $\omega$ defined by $\omega_{t}(s) = \omega(t+s)$, with $s \in [-r, 0]$ and the functions $\displaystyle g^{1}: L^{2}[0,\pi] \longrightarrow L^{2}[0,\pi],\, G_{i}^{1}:(t_{i}, s_{i}] \times \mathcal{X} \times \mathcal{U} \longrightarrow L^{2}[0,\pi]$ and $f^{1}:[0, T] \times \mathcal{PW} \times \mathcal{U} \longrightarrow L^{2}[0,\pi],$ are defined by
\begin{equation*}
\begin{array}{lll}
g^{1}(\omega_{t}(-r))(x)&=& g(\omega(t-r,x)),\\[2mm]
f^{1}(t, \omega_{t}(-r), u)(x)&=& f(t,\omega(t-r), u(t,x)),\\[2mm]
G_{i}^{1}(t, \omega(t), u(t))(x)&=& G_{k}(t, \omega(t,x), u(t,x)) \quad \textit{for}\quad i=0,...,N,
\end{array}
\end{equation*}
where $\mathcal{PW}$ is  the space of piecewise continuous functions given by
\begin{equation*}
\mathcal{PW}=\left\{h:[-r, 0] \longrightarrow \mathcal{X}^{1/2}: h \quad \text { is  piecewise continuous }\right\},
\end{equation*}
endowed with the norm
$$
\|h \| = \max \{\|h(t) \|_{\mathcal{X}}: -r \leq t \leq 0 \}.
$$

Next, we introduce the following function
$$
f^{2}:[0, T] \times \mathcal{PW} \times \mathcal{U} \longrightarrow \mathcal{X},
$$
given by
\begin{equation*}
\begin{array}{lll}
\displaystyle f^{2}(t,\omega,u) &=& \displaystyle \int_{0}^{t} M(t,s) g^{1}(\omega_{s}(-r)) \,\mathrm{d}s\, + f^{1}(t,\omega_{t}(-r),u(t)).
\end{array}
\end{equation*}
Then, from system \eqref{2.1},  we obtain the following non-autonomous differential equation with non-instantaneous impulses
\begin{empheq}[left = \empheqlbrace]{alignat=2}
\begin{aligned}\label{2.2}
&\frac{\partial \omega}{\partial{t}} + A \omega = \mathbf{B}_{\theta} u+f^{2}(t,\omega,u)   && \text{in } \bigcup\limits_{i=0}^{N}\left(s_{i}, t_{i+1}\right],\\
&\omega(s) = h(s), &&\text{in}\quad [-r,0],\\
&\omega(t,x)=  G_{i}^{1}(t,\omega(t),u(t)),  && \text{in}\quad \bigcup\limits_{i=0}^{N}\left(t_{i}, s_{i}\right].
\end{aligned}
\end{empheq}
We consider the space $\mathcal{P C}(\mathcal{X})$  of all functions $\varphi:\,\, [-r,T] \longrightarrow \mathcal{X} $ such that $\varphi(\cdot)$ is piecewise continuous on $[-r, 0]$ and continuous on $[0,T]$ except at points $t_{i}$ where the side limits $\varphi (t_{i}^{-})$ and  $\varphi (t_{i}^{+})$ exist, and $\varphi (t_{i}^{-}) = \varphi (t_{i})$ for all $i = 1,2,...,N,$ endowed with the uniform norm denoted by $\|\cdot\|_{\mathcal{P C}(\mathcal{X})}.$

\begin{df}
A function $\omega(\cdot)\in \mathcal{P C} (\mathcal{X}) $ is called a mild solution for the system \eqref{2.2}  if it satisfies the following integral-algebraic equation
\begin{equation}\label{Mild}
\omega(t)= \left\{
\begin{array}{lll}
  h(t), &t\in [-r,0], &\\[3mm]
 S(t)h(0) +\displaystyle \int_{0}^{t} S(t-s) \big(\mathbf{B}_{\theta} u(s)+ f^{2}(s,\omega,u(s))\big)\,\mathrm{d}s,  &t \in [0,t_{1}],&\\[3mm]
 G_{i}^{1}(t,\omega(t),u(t)), &t \in (t_{i},s_{i}], i\in I_{N},\\[3mm]
\displaystyle S(t-s_{i})G_{i}^{1}(s_{i},\omega(s_{i}),u(s_{i})) + \int_{s_{i}}^{t} S(t-s) \mathbf{B}_{\theta} u(s)\,\mathrm{d}s &t \in\left(s_{i}, t_{i+1}\right],\, i\in I_{N},& \\
 +\displaystyle \int_{s_{i}}^{t} S(t-s) f^{2}(s,\omega,u(s))\,\mathrm{d}s.
\;\;\;\;\\
\end{array}
\right.
\end{equation}
where $I_{N}$ denote the set $\{1,\cdots,N\}$.
\end{df}
The problem of the existence of a solution for semi-linear differential system under non-instantaneous impulses and delays has attracted many researchers. For instant, In finite dimensional Banach space the existence and uniqueness of solutions for semi-linear differential system under non-instantaneous impulses and delays are obtained by applying  Karakostas' fixed point theorem, see \cite{Zouhair2}, and for infinite dimension space we invite interested readers to see  \cite{Zouhair,Pierri,AK,HO}.

In this paper, we are interested in proving that the semilinear heat equation with non instantaneous impulses, memory and delay \eqref{1.1}  is approximately controllable on $[0,T]$. In this regard, we assume for the rest of this paper that $M \in L^{\infty}\big([0,T] \times [0,\pi]\big),$ and the nonlinear functions $f^{2}, g^{1}$ and $ G^{1}_{i}$ are smooth enough so that, for all $h \in \mathcal{PW}$ and $u \in L^{2}\left([0, T] ; \mathcal{U}\right)$ the problem \eqref{2.2} admit only one mild solution on $[-r,T]$.
\section{Approximate Controllability of the Linear Equation}\label{sec3}
In this section, we shall present some characterization of the approximate controllability for a general linear system in Hilbert spaces, then we prove for the better understanding of the readers the approximate controllability of the linear heat equation in any interval $[T-l, T]$, $l>0$ using the representation of the semigroup $(S(t))_{t\geq 0}$ generated by $A,$ and the fact that $\phi_{n}(\xi)= \sin{n\xi}$ are analytic functions. To this end, we note that, for all $\omega_{0} \in \mathcal{X},$ $0 \leq t_{0} \leq T$ and $u \in L^{2}(0,T;\mathcal{U})$ the initial value problem
\begin{equation}\label{3.1}
\left\{
\begin{array}{lll}
\displaystyle \frac{\partial \omega}{\partial{t}} = -A \omega +\mathbf{B}_{\theta} u, &\omega \in \mathcal{X},& \\\\
\displaystyle\omega(t_{0}) = \omega_{0},
\;\;\;\;\\
\end{array}
\right.
\end{equation}
\\
admits only one mild solution given by

\begin{equation*}
\omega(t)= S \left(t-t_{0}\right) \omega_{0}+\int_{t_{0}}^{t} S(t-s) B_{\theta} u(s)\, \,\mathrm{d}s, \quad t \in\left[t_{0}, T\right].
\end{equation*}
\begin{df}
For $l \in [0,T)$  we define the controllability map for the system \eqref{3.1} as follows
\begin{equation*}
\begin{array}{rl}
& G_{Tl}: L^{2}(T-l, T ; \mathcal{U})  \longrightarrow \mathcal{X} \\
&\displaystyle  G_{Tl}(v) = \int_{T-l}^{T} S(T-s) B_{\theta} v(s) \,\mathrm{d}s.
\end{array}
\end{equation*}
It's adjoint operator
\begin{equation*}
\begin{aligned}
&G_{T l}^{*}: \mathcal{X}  \longrightarrow L^{2}(T-l, T ; \mathcal{U}) \\[2mm]
& G_{T l}^{*}(x) = B_{\theta}^{*} S^{*}(T-t) x, \quad t \in[T-l, T].
\end{aligned}
\end{equation*}
Therefore, the Grammian operator $Q_{T l}: \mathcal{X}  \longrightarrow  \mathcal{X}  $ is defined by
\begin{equation*}
Q_{Tl}=G_{T l} G_{T l}^{*}=\int_{T- l}^{T} S(T-t) B_{\theta} B_{\theta}^{*} S^{*}(T-t) \,\mathrm{d} t.
\end{equation*}
\end{df}
The following lemma holds in general for a linear bounded operator $G: W \longrightarrow Z$ between Hilbert spaces $W$ and $Z.$
\begin{lem}(see \cite{CP,CZ,LMS})
The equation \eqref{3.1} is approximately controllable on $[T-l, T ]$ if, and only if, one of the following statements holds:
\\[1.5mm]
a.$ \quad \overline{\operatorname{Rang}\left(G_{T l}\right)}=\mathcal{X} ,$\\[1.5mm]
b.$ \quad B_{\theta} ^{*}S^{*}(T-t)x={0}, \quad t \in [T-l,T]  \longrightarrow x =0,$\\[1.5mm]
c. $\left\langle Q_{T l} x, x\right\rangle>0,\quad x \neq 0 \quad$ in $\quad \mathcal{X},$\\[1.5mm]
d. $\lim _{\alpha \longrightarrow 0^{+}} \alpha\left(\alpha I+Q_{T l}\right)^{-1} x=0, \quad \forall\,\, x \in \mathcal{X}.$
\end{lem}
\begin{rmq}
The Lemma $3.1$ implies that for all $x \in \mathcal{X}$ we have $G_{T l} u_{\alpha}=x-\alpha\left(\alpha I+Q_{T l}\right)^{-1} x,$ where
$$
u_{\alpha}=G_{T l}^{*}\left(\alpha I+Q_{T l}\right)^{-1} x, \quad \alpha \in(0,1].
$$
So, $\lim _{\alpha \longrightarrow 0} G_{T l} u_{\alpha}=x,$ and the error $E_{T l}x $ of this approximation is given by
$$
E_{T l} x=\alpha\left(\alpha I+Q_{T l}\right)^{-1} x, \quad \alpha \in(0,1],
$$
and the family of linear operators\\
$\Gamma_{\alpha T l}: \mathcal{X} \longrightarrow L^{2}(T-l, T ; \mathcal{U}),$ defined for $0<\alpha\leq 1$ by
$$
\Gamma_{\alpha T l} x = G_{T l}^{*}\left(\alpha I+Q_{T l}\right)^{-1} x,
$$
satisfies the following limit
$$
\lim _{\alpha \longrightarrow 0} G_{T l} \Gamma_{\alpha T l}=I,
$$
in the strong topology.
\end{rmq}
\begin{lem}
The linear heat equation \eqref{3.1} is approximately controllable on $[T-l,T ].$ Moreover, a sequence of controls steering the system \eqref{3.1} from an initial state $y_{0}$ to an $\varepsilon$ neighborhood of the final state $\omega^{1}$ at time $T > 0 ,$ is given by $\left\{u_{\alpha}^{l}\right\}_{0<\alpha \leq 1} \subset L^{2}(T-l, T ; \mathcal{U}),$ where
$$u_{\alpha}^{l}=G_{T l}^{*}\left(\alpha I+Q_{Tl}\right)^{-1}\left(w^{1}-S(l) y_{0}\right), \quad \alpha \in(0,1],$$
and the error of this approximation $E_{\alpha}$ is given by
$$E_{\alpha} = \alpha\left(\alpha I+Q_{T l}\right)^{-1} (w^{1}-S(l) y_{0}), $$

such that the solution $y(t)=y\left(t, T-l, y_{0}, u_{\alpha}^{l}\right)$ of the initial value problem
\begin{equation}\label{3.2}
\left\{\begin{array}{l}
y^{\prime}=-A y+B_{\theta} u_{\alpha}^{l}(t), \quad y \in \mathcal{X}, \quad t>0, \\
y(T-l)=y_{0},
\end{array}\right.
\end{equation}
satisfies
$$\lim _{\alpha \longrightarrow 0^{+}} y_{\alpha}^{l}\left(T, T-l, y_{0}, u_{\alpha}^{l}\right)=\omega^{1},$$
that is
$$\lim _{\alpha \longrightarrow 0^{+}} y_{\alpha}^{l}(T)=\lim _{\alpha \longrightarrow 0^{+}}\left\{S(l) y_{0}+\int_{T-l}^{T} S(T-s) B_{\theta} u_{\alpha}^{l}(s) \mathrm{d} s\right\}=\omega^{1}.$$

\end{lem}
\begin{proof}
We shall apply condition (b) from the foregoing Lemma. In fact, It is clear that $S^{*}(t) = S(t)$, $B_{\theta} ^{*} = B_{\theta}$. We suppose that
$
B_{\theta} ^{*}S^{*}(\tau-t)\xi={0}, \quad t \in [T-l,T],
$
which means,
$$
\sum_{n = 1}^{\infty} e^{-n^2(T-t)} <\xi, \phi_{n}> B_{\theta} \phi_{n}=0, \quad t \in [T-l,T],
$$
then,
$$
\sum_{n = 1}^{\infty} e^{-n^2(T-t)} <\xi, \phi_{n}> \mathbf{1}_{\theta} \phi_{n}=0, \quad t \in [T-l,T].
$$
hence,
$$
\sum_{n = 1}^{\infty} e^{-n^2(T-t)} <\xi, \phi_{n}> \phi_{n}(x)=0, \quad t \in [T-l,T], \quad  x \in \theta.
$$
then,
$$
\sum_{n = 1}^{\infty} e^{-n^2 t} <\xi, \phi_{n}> \phi_{n}(x)=0, \quad t \in [0,l], \quad  x \in \theta.
$$
From Lemma 3.14 from \cite{CP}, we get that
$$
<\xi, \phi_{n}> \phi_{n}(x)=0, \quad x \in \theta.
$$
Now, since $\phi_{n}(x) = \sin(nx)$ are analytic functions , we get that $<\xi, \phi_{n}>\phi_{n}(x) =0, \quad \forall x \in [0,\pi], \quad n= 1,2, \dots; $. This implies that
$$
<\xi, \phi_{n}> =0, \quad n= 1,2, \dots.
$$
 Since $\{\phi_{n}\}$ is a complete orthonormal set on $\mathcal{X}$, we conclude that $\xi=0$. This completes the proof of the approximate controllability of the linear system \eqref{3.1}.
The remaind of the prove follows from the above characterization of  dense range operators.
\end{proof}
\section{Approximate Controllability of the Semilinear System}\label{sec4}
In this section, we shall prove the main result of this paper, the interior approximate controllability of the non instantaneous impulsive semi-linear heat equation with memory and delay \eqref{1.1}.

For all $h \in \mathcal{PW}$ and $u \in L^{2}([0, T] ;\mathcal{U}),$ the initial value problem  \eqref{2.2} admit only one mild solution given by (\ref{Mild}), and its evaluation in $T$ leads us to the following expression
\begin{equation*}
\begin{array}{lll}
\displaystyle \omega(T) &=& S(T-s_{N})G_{N}^{1}(s_{N},\omega(s_{N}),u(s_{N})) + \displaystyle\int_{s_{N}}^{T} S(T-s) \bigg( \mathbf{B}_{\theta} u(s)+ f^{2}(s,\omega,u(s))\bigg)\mathrm{d} s \\[3mm]
&=&  S(T-s_{N})G_{N}^{1}(s_{N},\omega(s_{N}),u(s_{N})) + \displaystyle\int_{s_{N}}^{T} S(T-s) \mathbf{B}_{\theta} u(s)\mathrm{d} s\\
&+&\displaystyle \int_{s_{N}}^{T}S(T-s) \bigg( \displaystyle \int_{0}^{s} M(s,m) g^{1}(\omega_{m}(-r)) \mathrm{d} m+ f^{1}(s,\omega_{s}(-r),u(s))\bigg) \mathrm{d} s.
\end{array}
\end{equation*}
Now, we are ready to present and prove the main result of this paper.
\begin{thm}
Assume the existence of a function $\rho \in \mathcal{C} \left(\mathbb{R}_{+}, \mathbb{R}_{+} \right)$ which for all $(t, \Phi,u) \in [0, T] \times \mathcal{PW}(-r,0;  \mathcal{X}) \times L^{2}([0, T] ;\mathcal{U}),$ the following inequality holds
\begin{equation}\label{4.2}
\norm{f^1(t,\Phi,u)}_{\mathcal{X}}  \leq    \rho(\| \Phi \|_{\mathcal{X}}).
\end{equation}
Then, the non instantaneous impulsive semilinear heat eqution \eqref{1.1} with memory and delay is approximately controllable on $[0, T].$

\end{thm}
\begin{proof}
Given $\varepsilon >0,$ $h \in \mathcal{PW} $ and a final state $w^{1} \in \mathcal{X}$, we want to find a control $u_{\alpha}^{l} \in L^{2}(0, T ; \mathcal{U})$ such that
\begin{equation*}
\left\|\omega^{\alpha,l}(T)-\omega^{1}\right\|_{\mathcal{X}} <\varepsilon.
\end{equation*}
We start by considering $u \in L^{2}(0, T ; \mathcal{U})$ and its corresponding mild solution $\omega(t) = \omega(t,0,h,u),$ of the initial value problem \eqref{2.2}. For $0< \alpha <1 $
 and $0<l<\min\lbrace T-s_{N}, r\rbrace$  small enough. We define the control sequence $u_{\alpha}^{l} \in L^{2}(0, T ; \mathcal{U})$ as follows
 \begin{equation}
 u_{\alpha}^{l}(t)=\left\{\begin{array}{ll}
u(t), & \text { if } \quad 0 \leq t \leq T-l, \\
u_{\alpha}(t), & \text { if } \quad T-l<t \leq T,
\end{array}\right.
\end{equation}
where
\begin{equation}
u_{\alpha}(t)=B_{\theta}^{*} S^{*}(T-t)\left(\alpha I+Q_{T l}\right)^{-1}\left(\omega^{1}-S(l) \omega(T-l)\right), \quad T-l<t \leq T.
\end{equation}
The corresponding solution $\omega^{\alpha,l}=\omega(t,s_{N},G_{N}^{1},u_{\alpha}^{l})$ of the initial value problem \eqref{2.2} at time $T$ can be written as follows.
\begin{equation*}
\begin{array}{lll}
 \omega^{\alpha ,l}(T) &=&  S(T-s_{N})G_{N}^{1}(s_{N},\omega^{\alpha ,l}(s_{N}),u_{\alpha}^{l}(s_{N})) + \displaystyle\int_{s_{N}}^{T} S(T-s) \bigg[ \mathbf{B}_{\theta} u_{\alpha}^{l}(s)\\[3mm]
&+& \displaystyle \int_{0}^{s} M(s,m) g^{1}(\omega^{\alpha ,l}_{m}(-r)) \mathrm{d} m+ f^{1}(s,\omega^{\alpha ,l}_{s}(-r),u_{\alpha}^{l}(s))\,\bigg] \mathrm{d} s\\[3mm]
&=&S(l) \bigg\lbrace S(T-s_{N}-l) \,\,G_{N}^{1}(s_{N},\omega^{\alpha ,l}(s_{N}),u_{\alpha}^{l}(s_{N})) + \displaystyle\int_{s_{N}}^{T-l} S(T-s-l)\\[3mm]
&+&\bigg[ \mathbf{B}_{\theta} u_{\alpha}^{l}(s) \displaystyle \int_{0}^{s} M(s,m) g^{1}(\omega^{\alpha ,l}_{m}(-r)) \mathrm{d} m\,+ f^{1}(s,\omega^{\alpha ,l}_{s}(-r),u_{\alpha}^{l}(s))\,\bigg] \mathrm{d} s \bigg\rbrace\\[3mm]
&+&\displaystyle\int_{T-l}^{T} S(T-s)\bigg[ \mathbf{B}_{\theta} u_{\alpha}(s) + \displaystyle \int_{0}^{s} M(s,m) g^{1}(\omega^{\alpha ,l}_{m}(-r)) \mathrm{d} m\\
&+& f^{1}(s,\omega^{\alpha ,l}_{s}(-r),u_{\alpha}(s))\bigg] \mathrm{d} s,
\end{array}
\end{equation*}
then
\begin{equation*}
\begin{array}{lll}
 \omega^{\alpha ,l}(T)&=& S(l)\omega(T-l)+\displaystyle\int_{T-l}^{T} S(T-s)\bigg[ \mathbf{B}_{\omega}
u_{\alpha} + \displaystyle \int_{0}^{s} M(s,m) g^{1}(\omega^{\alpha ,l}_{m}(-r)) \mathrm{d} m\\[3mm]
&+& f^{1}(s,\omega^{\alpha ,l}_{s}(-r),u_{\alpha}(s))\bigg] \mathrm{d} s.
\end{array}
\end{equation*}
On the other hand, the corresponding solution $y_{\alpha}^{l}(t)=y\left(t, T-l , \omega(T-l), u_{\alpha}\right)$ of the linear value problem \eqref{3.2} at time T is given by
\begin{equation}
y_{\alpha}^{l}(T)=S(l) \omega(T-l)+\int_{T-l}^{T} S(T-s)\mathbf{B}_{\omega} u_{\alpha}(s) \mathrm{d} s.
\end{equation}
Therefore,
\begin{equation*}
\omega^{\alpha ,l}(T) -y_{\alpha}^{l}(T)=\displaystyle\int_{T-l}^{T} S(T-s)\bigg[\displaystyle \int_{0}^{s} M(s,m) g^{1}(\omega^{\alpha ,l}_{m}(-r)) \mathrm{d} m
+ f^{1}(s,\omega^{\alpha ,l}_{s}(-r),u_{\alpha}(s))\bigg] \mathrm{d} s,
\end{equation*}
by the hypothesis \eqref{4.2}, we obtain
\begin{equation*}
\begin{aligned}
\left\|\omega^{\alpha ,l}(T) -y_{\alpha}^{l}(T)\right\| \leq& \int_{T-l}^{T}\|S(T-s)\| \int_{0}^{s}\left\|M(s, m) g^{1}\left(\omega^{\alpha ,l}_{m}(-r) \right)\right\| \mathrm{d} m\, \mathrm{d} s\\
&+ \int_{T-l}^{T}\|S(T-s)\| \rho\left(\left\|\omega^{\alpha ,l}_{s}(-r)\right\|\right) d s,
\end{aligned}
\end{equation*}
since $0\leq m \leq s,$ $0<l<r,$ and $T-l\leq s \leq T,$ then  $m-r \leq s-r \leq T-r< T-l,$ then
\begin{equation}
\omega^{\alpha ,l}(m-r)=\omega(m-r) \text { and } \omega^{\alpha ,l}_{s}(-r)=\omega(s-r),
\end{equation}
therefore, for a sufficiently small $l$ we obtain
\begin{equation*}
\begin{aligned}
\left\|\omega^{\alpha ,l}(T) -y_{\alpha}^{l}(T)\right\| \leq& \int_{T-l}^{T}\|S(T-s)\| \int_{0}^{s}\left\|M(s, m) g^{1}\left(\omega_{m}(-r)\right)\right\| \mathrm{d} m\, \mathrm{d} s\\
&+ \int_{T-l}^{T}\|S(T-s)\| \rho\left(\left\|\omega(s-r)\right\|\right) \mathrm{d} s \\
&\leq \frac{\varepsilon}{2}.
\end{aligned}
\end{equation*}
Hence, by \textit{lemma 3.2} we can choose $\alpha >0$ such that
\begin{equation*}
\begin{aligned}
\left\|\omega^{\alpha ,l}(T)-\omega^{1}\right\| & \leq\left\|\omega^{\alpha ,l}(T) -y_{\alpha}^{l}(T)\right\|+\left\|y_{\alpha}^{l}(T)-\omega^{1}\right\| \\
&<\frac{\varepsilon}{2}+\frac{\varepsilon}{2} = \varepsilon.
\end{aligned}
\end{equation*}

\end{proof}

\begin{obs}
In particular of function $\rho$ from condition \eqref{4.2}, is $\rho(\xi)=e(\xi)^\beta + \eta$, with $\beta \geq 1$.
\end{obs}
\subsection{Open Problem}
Our first open problem is the  semilinear Non-autonomous  differential equations with non instantaneous impulses, memory and delay
\begin{equation*}\label{5.1}
\left\{
\begin{array}{lll}
\displaystyle \omega^{'}(t) = A(t) \omega(t) +B(t) u(t)+ \int_{0}^{t} M(t,s) g(\omega_{s}) ds+ f(t,\omega_{t},u(t))\,\mathrm{d} s, &in& \bigcup\limits_{i=0}^{N}\left(s_{i}, t_{i+1}\right], \\
\displaystyle\omega(s,x) = h(s,x), &in& [-r,0], \\
\displaystyle\omega(t)= G_{i}(t,\omega(t),u(t)),  &in& \bigcup\limits_{i=0}^{N} (t_{i},s_{i}],\;\;\;\;
\end{array}
\right.
\end{equation*}
where $A(t),$ $B(t)$ are continuous matrices of dimension $n\times n$ and $n\times m$ respectively, the control function $u$ belongs to $\mathcal{C}(0,T;\R^{m}) ,$ $ h \in \mathcal{PW}(-r,0;\R^{n}),$ $f :[0,T]\times \mathcal{PW}(0,T;\R^{n})\times \R^{m} \longrightarrow \R^{n} ,$ $g : \R^{n} \longrightarrow \R^{n}, $ $ M \in \mathcal{C}(0,T;\R^{n}) $ $G_{i} :(t_{i},s_{i}]\times \R^{n}\times \R^{m} \longrightarrow \R^{n} .$
\subsection{Open Problem}
Second open problem is about Controllability of non instantaneous   semilinear beam equation with memory and delay in a bounded domain $\Omega \subseteq\mathbb{R}^{N},$
\begin{equation*}
\left\{\begin{array}{ll}
z^{''}-2 \gamma \Delta z^{'}+\Delta^{2} z=u(t, x)+f\left(t, z(t-r), z^{'}(t-r), u\right) &\text{in}\,\, \bigcup\limits_{i=0}^{N}\left(s_{i}, t_{i+1}\right],\\
\displaystyle+\int_{0}^{t} g(t-s) h(z(s-r, x)) \mathrm{d} s,\\
z(t, x)=\Delta z(t, x)=0, & \text{on}\,\, \partial \Omega, \\
z(s, x)=\varphi_{1}(s, x), & \text{in}\,\,[-r, 0], \\
z^{'}(s, x)=\varphi_{2}(s, x), \\
\displaystyle z^{'}= \psi_{i}(t,z(t),z^{'}(t),u(t)),  &\text{in}\,\, \bigcup\limits_{i=0}^{N}\left(t_{i}, s_{i}\right],
\end{array}\right.
\end{equation*}
the damping coefficient $\gamma > 1,$ and the
real-valued functions  $ z = z(t,x) $ in $[0,T] \times \Omega$ represents the beam deflection, $u$ in  $[0,T] \times \Omega$ is the distributed control, $g$ acts as convolution kernel with respect to the time variable.
\subsection{Open Problem}
Another example where this technique may be applied is the
strongly damped wave equation with Dirichlet boundary conditions with non instantaneous impulses, memory and delay in
\begin{equation*}
\left\{\begin{array}{ll}
\displaystyle y^{''} +\beta(-\Delta)^{1/2} y^{'}  + \gamma (-\Delta) y = \mathbf{1}_{\theta}u+\displaystyle\int_{0}^{t} h(s,y(s-r),u(s)) \mathrm{d} s, &\text{in}\,\, \bigcup\limits_{i=0}^{N}\left(s_{i}, t_{i+1}\right],\\
y = 0, & \text{on}\,\, \partial \Omega ,\\
y(s)=\phi_{1}(s), &\text{in}\,\,[-r, 0], \\
y^{'}(s)=\phi_{2}(s), &\text{in}\,\,[-r, 0], \\
\displaystyle y^{'}(t)= g_{i}(t,y(t),y^{'}(t),u(t)),&\text{in}\,\,\bigcup\limits_{i=0}^{N} (t_{i},s_{i}].
\end{array}\right.
\end{equation*}
\\
In the space $\mathcal{Z}_{1/2} = \mathcal{D}( (-\Delta)^{1/2}) \times L^{2}(\Omega),$ where $\Omega$ is a bounded domain in $\mathbb{R}^{N},$ $\theta$ is an open nonempty subset of $\Omega,$ $\mathbf{1}_{\theta}$ denotes the characteristic function of the set $\theta,$ the distributed control $u \in L^{2}(0,T; L^{2}(\Omega)). $ $\phi_{1},\phi_{2}$ are piecewise continuous functions.

\section*{Acknowledgment}
The authors would like to express their thanks to the editor and anonymous referees for constructive comments and suggestions that improved the quality of this manuscript, they also declare the presentation of the manuscript as a preprint \cite{Zouhair3}.

\end{document}